\documentclass[12pt]{iopart}

\usepackage{iopams}  
\usepackage[utf8]{inputenc}
\usepackage[comma, numbers]{natbib}
\usepackage{xcolor}
\usepackage{graphicx}
\usepackage{xspace}

\expandafter\let\csname equation*\endcsname\relax
\expandafter\let\csname endequation*\endcsname\relax
\usepackage{amsmath}
\usepackage{amsthm}

\newcommand{\mydiv}{\textrm{div}}

\newcommand{\norm}[1]{\left\lVert#1\right\rVert}
\newcommand{\Norm}[2]{\left( \left\lVert#1\right\rVert^2 + \left\lVert#2\right\rVert^2 \right)^{1/2}\xspace}
\newcommand{\NormTri}[3]{\left( \left\lVert#1\right\rVert^2 + \left\lVert#2\right\rVert^2 + \left\lVert#3\right\rVert^2 \right)^{1/2}\xspace}

\newcommand{\myat}[1]{\left. #1 \right|}

\newtheorem{proposition}{Proposition}

\begin{document}

\title[]{PDE-constrained optimization for electroencephalographic source reconstruction}

\author{M S Malovichko$^{1}$, N B Yavich$^{1}$, A M Razorenova$^{2,3}$, N A Koshev$^2$}
\address{$^1$ Moscow Institute of Physics and Technology, Dolgoprudny, Russia}
\address{$^2$ Skolkovo Institute of Science and Technology, Moscow, Russia}
\address{$^3$ Center for Neurocognitive Research (MEG Center), Moscow State University of Psychology and Education, Moscow, Russia}
\ead{malovichko.ms@mipt.ru}

\begin{abstract}

This paper introduces a novel numerical method for the inverse problem of electroencephalography(EEG).
We pose the inverse EEG problem as an optimal control (OC) problem for Poisson's equation. 
The optimality conditions lead to a variational system of differential equations. 
It is discretized directly in finite-element spaces
leading to a system of linear equations with a sparse Karush-Kuhn-Tucker matrix.
The method uses finite-element discretization and thus can handle MRI-based meshes of almost arbitrary complexity.
It extends the well-known mixed quasi-reversibility method (mQRM) in that pointwise noisy data explicitly appear in the formulation making unnecessary tedious interpolation of the noisy data from the electrodes to the scalp surface.
The resulting algebraic problem differs considerably from that obtained in the mixed quasi-reversibility, but only slightly larger.
An interesting feature of the algorithm is that it does not require the formation of the lead-field matrix. 
Our tests, both with spherical and MRI-based meshes, demonstrates that the method accurately reconstructs cortical activity.

\end{abstract}

\date{\today}

%\maketitle

\section{Introduction}
\label{sec1}

Electroencephalography (EEG) has a long history of development and
constitute one of the major modalities to study living human brain, see reviews 
\cite{Baillet_2001,Michel_2004,Soufflet_2005,Grech_2008,PascualMarqui_2009,Niedermeyer_EEG,Ahlfors_2012}.
For reliable localization of cortical activations, it is crucial to solving the inverse EEG problem, that is, to map electric potential measured at electrodes to either potential or electric current on the cortex.

Mathematically, this problem is the source reconstruction for Poisson's equation.
Depending on the formulation, it can also be considered as the Cauchy problem.
The mathematical properties of these problems have been extensively studied, and many approaches have been proposed.
We refer to monographs 
\cite{Lavrentyev,Isakov,Prilepko2000,Klibanov2000} 
as well research papers \cite{Klibanov1991,Kozlov1991,Kabanikhin1995,Cimetiere2001,Hao2018, Leonov2020} (among many others).
Because it is impossible to review all methods proposed for the task, 
we confine ourselves to those that found applications in the EEG context.

\textit{The linear distributed estimators} constitute arguably the most popular approach. 
They are used in practice with great success
e.g., \citep{Dale_1993,Fuchs_1999,Lin_2006}. 
Algorithms in this category, such as MNE-family \citep{Hamalainen_1994,Lin_2006}, LORETA-family \citep{PascualMarqui_1994,PascualMarqui_2002}, LAURA \citep{DePeraltaMenendez_2004}, FOCUSS \citep{Gorodnitsky_1995}, WROP \citep{PeraltaMenendez1997} and many others,
are closely connected to the Tikhonov regularization approach.
The forward operator is usually based on the boundary-element method (BEM), 
but the finite-element method (FEM) is catching up; see \cite{Yavich2021}, and references therein. 

Methods based on solving the Cauchy problem for Poisson's equation propagate the electric potential and, sometimes, Ohmic current inward the head starting from known values of potential and its normal derivative on the scalp.
These methods date back at least to the 1970s, under different names, such as 
\textit{the current source density}, \textit{surface Laplacian}, \textit{deblurring}, \textit{spatial deconvolution} and others 
\citep{Nicholson_1973,Nicolas_1976,Freeman_1980,Nunez_1993,Srinivasan_1996,Junghofer_1997,Tenke_2005}.
They solve the Cauchy problem for Laplace’s or Poisson's equations making severe simplifications to analytically link the potential and its normal derivatives on the scalp with that inside the head. 
These early techniques have been considered as data-enhancing procedures rather than rigorous solutions to an inverse problem. 
Today, they are still in use, sometimes as a part of more sophisticated 
algorithms \citep{Haor2017}. 

The deblurring method by \cite{Gevins_1990,Gevins_patent_1994,Gevins_patent_1996} was one of the first methods to propagate potential inward an anatomical head model. 
Remarkably, it employed a FEM on tetrahedral grids, although the BEM would dominate the field for decades. 
Numerical procedures for potential propagation based on the BEM have been proposed in \citep{He_1999, He_2002,Clerc_2007}. 
These approaches are appealing, although they have some limitations rooted in the BEM method, such as the inability to handle surface holes, 
and a high compute load for large grids.

Thus, despite the variety of mathematical approaches, the majority of source reconstruction methods currently used in EEG community are based 
on the construction and solution of a normal system of equations with the explicit use of the inverse operator.
A remarkable dispatch from this pattern has been made by L.~Bourgeois, who proposed the mixed quasi-reversibility method (QRM) in a purely mathematical context \citep{Bourgeois_2005,Bourgeois_2006}. 
The mixed QRM reduces the Cauchy problem to a variational system.
The system is discretized directly in finite-element function spaces.
Thus, reconstruction of the potential on the unaccessible boundary (cortex) requires a single linear solve.
The FE discretization makes the mixed QRM very flexible, allowing computational domains of almost arbitrary complexity.
This approach has been advanced further to account for the non-constant conductivity and adjusted to the EEG source localization in \citep{Koshev2020,Malovichko2020}.

The mixed QRM has some drawbacks. 
The main practical disadvantage is the need to convert pointwise measurements on electrodes to a boundary value of potential on the scalp. 
Accurate interpolation of noisy data to an irregular surface (scalp) is a tedious task of its own. 
Even more important is that it is unclear how to take proper account of noise characterization, 
which is typically available by a noise covariance matrix, estimated for a given electrode layout.
%Even more important is that this and similar methods are lesser studied theoretically than the linear distributed estimators (MNE, LORETA, etc.). 
%As a result of this lack of understanding, it is not easy to compare QRM with the classical linear estimates. For example, apriori is unknown which of the MNE or QRM gives a smoother solution, more robust to noise, etc.

In this paper, we study the inverse EEG problem from the optimal control (OC) standpoint. 
It is a well-established technique for inverse problems associated with PDE \cite{Hinze2009}, 
but it is rarely applied for the EEG/MEG source analysis.
Elements of the OC were used in \citep{Vallaghe_2009} as a tool to justify the calculation of the lead-field matrix through the adjoint-state technique.
In \citep{Faugeras_1999} the inverse EEG problem was formulated as a problem of finding the saddle point of the Lagrangian.
Those authors did not form the Euler-Lagrange system but approached the saddle point making alternating steps 
in the state and adjoint-state spaces; thus, this method never constituted an efficient algorithm.

Here, we formulate the inverse problem of EEG as an optimization problem with constraints, 
thus reducing it to the problem of finding an optimal point of Lagrangian.
We formulate the Euler-Lagrange variational system and approximate it in the finite-element function spaces.
Based on this approach, we proposed a new numerical method for the inverse problem of EEG, which uses pointwise noisy data.
This method addresses the critical issue of noise interpolation and, in this respect, 
can be considered as an “improved” mixed QRM.
Another advantage of our method is it eliminates the brain compartment from consideration, making the resulting mesh smaller and avoiding errors associated with distinguishing between white and grey matters. 
Also, since our algorithm does not form the lead-field matrix, the issues connected with interpolating a dipolar right-hand size in finite-element spaces are removed.
%More importantly, we revealed the relationship between Cauchy methods and linear distributed estimators. 
%In particular, we show that our approach is mathematically equivalent to classical methods (MNE, LORETA, etc.), but it leads to a very different algebraic problem.
%A practical conclusion to be drawn is that the two classes of methods provide very similar if not identical source estimations but require different amounts of computing resources.
%In summary, the main contribution of this article is that it brings together the theory of the potential propagation methods and linear distributed estimators based on the optimal-control formalism.

The paper is organized as follows. 
In Section~\ref{sec2}, we review the mixed quasi-reversibility method.
Section~\ref{sec3} is dedicated to the optimal-control formulation and derivation of the Euler-Lagrange variational system. 
In Section~\ref{sec4}, we verify the derived variational system.
The finite-element discretization, the algebraic problem, and the solution of the resulting system of 
linear equations are considered in Section~\ref{sec5}.
The FEM convergence is discussed in \ref{sec6}.
A numerical experiment aimed to verify our numerical method is presented in Section~\ref{sec7}.
Concluding remarks are given in Section~\ref{sec8}.
%%%%%%%%%%%%%%%%%%%%%%%%%%%%%%%%%%%%%%%%%%%%%%%%%%%%%%%%%%%

%%%%%%%%%%%%%%%%%%%%%%%%%%%%%%%%%%
\section{The mixed quasi-reversibility method}
\label{sec2}

Here we review the mixed QRM.
The QRM can be traced back to \cite{LionsOC}.
Its mixed formulation within the finite-element framework has been proposed in \citep{Bourgeois_2005,Bourgeois_2006}.
It was advanced further in the context of EEG source analysis in \citep{Koshev2020,Malovichko2020}.
We derive the algorithm with OC formalism.
This derivation is similar to one presented in \citep{Malovichko2020}, 
but is adjusted toward the formal Lagrangian approach, used throughout the rest of this paper.

Let $\Omega \in \mathbb{R}^3$ be a part of the head enclosed by the scalp surface, 
$\Gamma_0$, and the cortical surface, $\Gamma_B$, that is, $\partial \Omega=\Gamma_0 \cup \Gamma_B$ (Figure~\ref{fig:head-schematic})
Electric potential on the scalp has zero normal derivative, $\partial u / \partial \nu =0$.
We assume that the potential inside the head is driven by a dipole layer attached to the cortical surface.
Let us denote by $f$ the unknown surface density of dipoles.
\begin{figure}[!htb]
\centering
		\includegraphics[width=2.in]{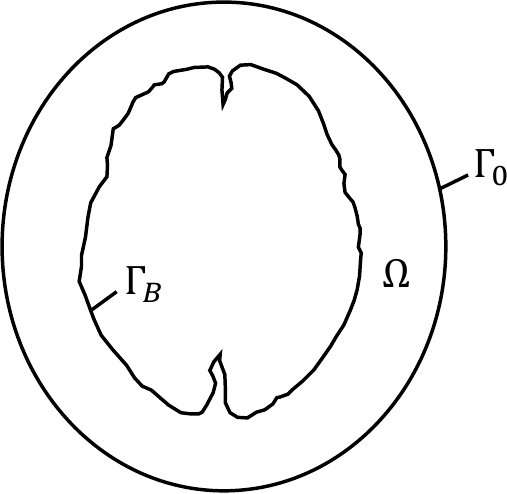}
	\caption
	{
		Schematic representation of the head. $\Gamma_0$ is the scalp surface, $\Gamma_B$ is the cortex (can consist of disjoint surfaces, e.g. the two hemispheres).
		Volume $\Omega$ includes different tissues (fat,skull, cerebrospinal fluid, etc.).
	}
\label{fig:head-schematic}
\end{figure}
Thus, the potential in the head satisfies the following differential equation,
\begin{subequations}
\label{eq:forward_problem}
\renewcommand{\theequation}{\theparentequation.\arabic{equation}}
    \begin{equation}
         -\mydiv(\sigma \nabla u) = 0, \text{ in } \Omega,   
    \end{equation}
    \begin{equation}
         \frac{\partial u}{\partial \nu} = 0, \text{ on } \Gamma_0,   
    \end{equation}
    \begin{equation}
         \sigma \frac{\partial u}{\partial \nu} = f, \text{ on } \Gamma_B.   
    \end{equation}
\end{subequations}

The head consists of several compartments of different conductivities.
The conductivity is usually considered smooth (even constant) inside each compartment 
but has jumps across the interfaces.
At the interfaces,  the potential is continuous, but the normal Ohmic current experiences jumps,
\begin{equation}
	\label{eq:interface_conditions}
	u_i=u_j, \quad \sigma_i \nabla u_i = \sigma_j \nabla u_j \text{ at } \Gamma_{ij},
\end{equation}
where $\Gamma_{ij}$ is the interface between a pair of compartments.
To avoid theoretical complications, we assume that the boundaries are smooth, 
which seems reasonable for EEG studies.
The problem \eqref{eq:forward_problem} is understood in the weak sense.
We specify $u \in H^1(\Omega)$, $\sigma \nabla u \in H_{div}(\Omega)$.
Here we used standard notation: $H^1(\Omega)$ is a space of functions having square-integrable gradient; 
$H_{div}(\Omega)$ is a space of functions having square-integrable divergence.
Thus, the interface conditions \eqref{eq:interface_conditions} not need to be added explicitly to 
\eqref{eq:forward_problem}, because they are satisfied by this choice
of function spaces.

We want to reconstruct potential $u$ in the head (including that on the cortex) 
such that 
(a) $u$ is satisfies Poisson's equation, 
(b) $u$ is smooth almost everywhere, 
and (c) $\myat{u}_{\Gamma_D}$ is equal to some given value $g_D$.
Consider the following inverse problem,
\begin{subequations}
	\renewcommand{\theequation}{\theparentequation.\arabic{equation}}
	\label{eq:qrm_continous}
	\begin{equation}	
		\label{eq:qrm_continous_1}
		\underset{u}{\text{minimize}} \quad
		\frac{\varepsilon}{2} \int_{\Omega} |\nabla u(f)|^2 \, dV,
	\end{equation}
	\begin{equation*}	
		\text{\textit{subject to}}
	\end{equation*}	
	\begin{equation}	
		\label{eq:qrm_continous_2}
		- \mydiv ( \sigma \nabla u )= 0, \quad \text{ in } \Omega,
	\end{equation}
	\begin{equation}	
		\label{eq:qrm_continous_3}
		\frac{\partial u}{\partial \nu} = 0, \text{ on } \Gamma_0,
	\end{equation}
	\begin{equation}	
		\label{eq:qrm_continous_4}
		u = g_D, \text{ on } \Gamma_0.
	\end{equation}
\end{subequations}
Here $\varepsilon>0$ is the regularization parameter, 
notation $u(f)$ emphasizes the fact that $u$ depends on $f$ 
through constraint \eqref{eq:qrm_continous_2}-\eqref{eq:qrm_continous_4}.

We form the Lagrangian that includes condition \eqref{eq:qrm_continous_2},
\begin{equation}	
	\label{eq:qrm_lagrangian}
	\mathcal{L}(u,\lambda) = 
	\frac{\varepsilon}{2} \int_{\Omega} |\nabla u|^2 \, dV
	-
	\int_{\Omega} \lambda \, \mydiv (\sigma \nabla u) \, dV,
\end{equation}
where $\lambda$ is the Lagrange multiplier.

By applying the Green-Gauss theorem to the second term of \eqref{eq:qrm_lagrangian} 
and using condition \eqref{eq:qrm_continous_3} 
we obtain the following equality,
\begin{equation}	
		\label{eq:qrm_green_gauss}
		\int_{\Omega} \lambda \, \mydiv (\sigma \nabla u) \, dV 
		=
		-\int_{\Omega} \sigma \nabla u \cdot \nabla \lambda  \, dV
		+
		\int_{\Gamma_B} \sigma \frac{\partial u}{\partial \nu} \lambda \, dV.
\end{equation}
Let us choose $\lambda$ such that its boundary values on $\Gamma_B$ is zero, 
so the boundary integral in \eqref{eq:qrm_green_gauss} vanishes,
\begin{equation}
	\label{eq:Q_definition}
	\lambda \in Q := \{ v \in H^1(\Omega) : \myat{v}_{\Gamma_B} = 0 \}.
\end{equation}
Now, we choose $u$ such that condition \eqref{eq:qrm_continous_4} is satisfied automatically by this choice:
\begin{equation}
	\label{eq:V_definition}
	u \in V := \{  v \in H^1(\Omega) : \myat{v}_{\Gamma_0} = g_D \}.
\end{equation}
Now, we substitute \eqref{eq:qrm_green_gauss} to \eqref{eq:qrm_lagrangian} using definitions 
\eqref{eq:Q_definition} and \eqref{eq:V_definition}, and get
\begin{equation}	
	\label{eq:qrm_lagrangian_new}
	\mathcal{L}(u,\lambda) = 
	\frac{\varepsilon}{2} \int_{\Omega} |\nabla u|^2 \, dV
	+
	\int_{\Omega} \sigma \nabla u \cdot \nabla \lambda  \, dV.
\end{equation}
The Lagrangian \eqref{eq:qrm_lagrangian_new} attains its optimal point whenever
$\mathcal{L}'_{u} = 0$ and $\mathcal{L}'_{\lambda} = 0$.
Thus, we have the following variational problem.
\begin{subequations}
	\renewcommand{\theequation}{\theparentequation.\arabic{equation}}
	\label{eq:qrm_variational}
	\begin{equation*}	
		\text{ Find } (u,\lambda) \in V \times Q, \text{ such that }
	\end{equation*}
	\begin{equation}	
		\label{eq:qrm_variational_1}
		\varepsilon \int_{\Omega} \nabla u \cdot \nabla v \, dV 
		+
		\int_{\Omega} \sigma \nabla \lambda \cdot \nabla v \, dV
		= 0,
		\quad  \forall v \in V_0,
	\end{equation}
	\begin{equation}	
		\label{eq:qrm_variational_2}
		\int_{\Omega} \sigma \nabla u \cdot \nabla \mu \, dV 
		= 0, \quad \forall \mu \in Q.
	\end{equation}
\end{subequations}
Function space $V_0$ is specified as a set of all elements from $H^1(\Omega)$ having zero boundary value on $\Gamma_0$,
\begin{equation}
	\label{eq:V0_definition}
	V_0 := \{  v \in H^1(\Omega) : \myat{v}_{\Gamma_0} = 0 \}.
\end{equation}

This problem needs to be stabilized, see \cite{Bourgeois_2005},Remark 3. 
We add term $-\delta \int_{\Omega} 
|\nabla \lambda|^2 dV$ to \eqref{eq:qrm_lagrangian_new}, where $\delta > 0$.
The modified variational problem reads,
\begin{subequations}
	\renewcommand{\theequation}{\theparentequation.\arabic{equation}}
	\label{eq:qrm_variational_new}
	\begin{equation*}	
		\text{ \textit{Find} } (u,\lambda) \in V \times Q \text{ \textit{such that} }
	\end{equation*}
	\begin{equation}	
		\label{eq:qrm_variational_new_1}
		\varepsilon \int_{\Omega} \nabla u \cdot \nabla v \, dV 
		+
		\int_{\Omega} \sigma \nabla \lambda \cdot \nabla v \, dV
		= 0,
		\quad  \forall v \in V_0,
	\end{equation}
	\begin{equation}	
		\label{eq:qrm_variational_new_2}
		\int_{\Omega} \sigma \nabla u \cdot \nabla \mu \, dV 
		-
		\delta \int_{\Omega} \nabla \lambda \cdot \nabla \mu \, dV
		= 0, \quad \forall \mu \in Q.
	\end{equation}
\end{subequations}
This is the mixed formulation of the quasi reversibility.
Note that, instead of $\int_{\Omega} |\nabla u|^2 dV$ and $\int_{\Omega}|\nabla \lambda|^2 dV$ we could have used 
$\int_{\Omega} (|u|^2 + |\nabla u|)^2 dV$ and $\int_{\Omega}(|\lambda|^2 + |\nabla \lambda|)^2 dV$, respectively. 
This way we would have variational system (14) from \citep{Malovichko2020}.

System \eqref{eq:qrm_variational_new} was rigorously studied in \citep{Bourgeois_2005,Bourgeois_2006}.
Its numerical discretization and validation in the EEG context were presented in \citep{Malovichko2020}.

Note that this algorithm requires as input the boundary value of $u$ on $\Gamma_0$, i.e., $g_D$.
In practice, we have discrete data: potential is measured at a small set of electrodes.
Interpolation of discrete data onto a whole surface of the grid produces considerable correlated noise.
More importantly, it is not clear how to incorporate the noise characterization, 
typically represented by a noise covariance matrix, into formulation 
\eqref{eq:qrm_variational_new}.
Finally, formulation \eqref{eq:qrm_variational_new} imposes a constraint 
on the volume distribution of $u$, although researchers prefer to control the cortical distribution of the potential.

%%%%%%%%%%%%%%%%%%%%%%%%%%%%%%%%%%
\section{The PDE-constrained optimization}
\label{sec3}

In this section, we apply the optimal control formalism to the EEG source reconstruction problem.

The electrical potential is measured in a set of electrodes. 
Each data point $d_i$ is a value of potential $u$ at the $i$-th electrode.
Formally, $d_i$ is the convolution of $u$ with the Dirac $\delta$ function,
\begin{equation}
	\label{eq:Qi_def}
    d_i = \mathcal{Q}_i(u) := \int_{\Omega} \delta(x-x_i) u(x) \, dV,
\end{equation}
where $x_i \in \mathbb{R}^3$ is the position of the $i$-th electrode, $\mathcal{Q}_i$ is the observation operator. Let $d \in \mathbb{R}^K$ be a vector of all data points.

We introduce a diagonal matrix of weights, $W$, representing estimates of noise standard deviation.
Each data point $d_i$ is associated with a weight $w_i>0$.
By assuming that $W$ is diagonal, that is, the noise is uncorrelated, we aim to simplify the formulas below.
A dense noise covariance matrix can be readily used here (see next section).

We set up the following optimization problem.
\begin{subequations}
\renewcommand{\theequation}{\theparentequation.\arabic{equation}}
\label{eq:minimization}
\begin{equation}
	\label{eq:minimization_1}
	\underset{f,u}{\text{minimize}} \quad
	\Psi(u,f) :=
	\frac{1}{2} \sum_{i=0}^{K-1} w_i^2(\mathcal{Q}_i u - d_i)^2 + 
	\frac{\varepsilon}{2} \int_{\Gamma_B}| \nabla_B f|^2 \,dS,
\end{equation}
	\text{ \textit{subject to} }
\begin{equation}
	\label{eq:minimization_2}
	-\mydiv (\sigma \nabla u) = 0, \text{ in } \Omega, 
\end{equation}
\begin{equation}
	\label{eq:minimization_3}
	 \frac{\partial u}{\partial \nu} = 0, \text{ on } \Gamma_0, 
	 \sigma \frac{\partial u}{\partial \nu} = f, \text{ on } \Gamma_B.
\end{equation}
\end{subequations}
The first term in \eqref{eq:minimization_1} is the misfit between measured data, $d_i$, and computed data, $Q_i u$.
The second term is the stabilizing functional with $\varepsilon >0$ is the regularization parameter.
Here $\nabla_B$ is the surface gradient defined on $\Gamma_B$.
Thus, we require $f$ be a weakly smooth function on $\Gamma_B$.
Constraints \eqref{eq:minimization_2} and \eqref{eq:minimization_3} determine the link between the electric potential, $u$,
and the source function, $f$.

Here we outline the functional framework for $u$ and $f$.
Since data $d$  are "pointwise", the integrals \eqref{eq:Qi_def} must exist for $\delta \in H^{s}, s<-3/2$ and any admissible state $u$.
Let us specify the space of admissible states as
\begin{equation}
	U_{ad} := \{ u \in H^1(\Omega) \cap \mathcal{C}(\bar{\Omega})\}.
\end{equation}
We also specify the space of admissible controls $f$,
\begin{equation}
	F := \{ f \in H^1(\Gamma_B): f \neq const\}.
\end{equation}

Next, we derive the variational system for problem \eqref{eq:minimization}.
The derivation is based on the Lagrangian formalism aiming to design the proper form of the Euler-Lagrange system.
We form the Lagrangian $\mathcal{L}$ as follows,
\begin{equation}
\label{eq:lagrangian1}
    \mathcal{L}(u,f,\lambda) = \Psi(u,f) - 
    \int_{\Omega} \mydiv (\sigma \nabla u) \, \lambda \, dV
    + \int_{\Gamma_B} \left( \sigma \frac{\partial u}{ \partial \nu} - f \right) \lambda \, dS ,
\end{equation}
where $\lambda$ is the Lagrange multiplier, defined in $\Omega \cup \Gamma_B$.
Next, we apply the Green-Gauss theorem to the second term of \eqref{eq:lagrangian1} and 
use condition $\partial u / \partial \nu = 0$ on $\Gamma_0$:
\begin{equation}
\label{eq:green-gauss}
    \int_{\Omega} \mydiv (\sigma \nabla u) \, \lambda \, dV = 
    - \int_{\Omega} \sigma \nabla u \cdot \nabla \lambda \, dV
    + \int_{\Gamma_B} \sigma \frac{\partial u}{\partial \nu} \lambda \, dS .
\end{equation}
We have
\begin{equation}
\label{eq:lagrangian2}
    \mathcal{L}(u,f,\lambda) = \Psi(u,f) 
    + \int_{\Omega} \sigma \nabla u \cdot \nabla \lambda \, dV
    - \int_{\Gamma_B}f  \lambda \, dS ,    
\end{equation}
Thus, optimization with constraints \eqref{eq:minimization} translates to the problem of finding a stationary point of the Lagrangian \eqref{eq:lagrangian2}.
The necessary conditions for a triple $(u,f,\lambda)$ be a stationary point are that the derivatives of $\mathcal{L}(u,f,\lambda)$ with respect to $u$,$f$, and $\lambda$ vanish.
Since $\Psi$ is convex, the necessary conditions are also sufficient.
Thus, we set up the following variational problem,
\begin{subequations}
\label{eq:variational_system1}
\renewcommand{\theequation}{\theparentequation.\arabic{equation}}
\begin{equation*}
	\text{\textit{Find triple }} 
	(f,\lambda,u)\in F \times \Lambda \times U
	\text{\textit{ which verifies }}
\end{equation*}
\begin{equation}
	\label{eq:variational_system1-1}
    \varepsilon \int_{\Gamma_B} \nabla_B f \cdot \nabla_B \varphi \, dS
    - \int_{\Gamma_B} \lambda \varphi \, dS = 0
	, \quad \forall \varphi \in F,
\end{equation}
\begin{equation}
	\label{eq:variational_system1-2}
    -\int_{\Gamma_B} f \mu \, dS
    +\int_{\Omega} \sigma \nabla u \cdot \nabla \mu \, dV  = 0
	, \quad \forall \mu \in \Lambda,
\end{equation}
\begin{equation}
	\label{eq:variational_system1-3}
	\int_{\Omega} \sigma \nabla \lambda \cdot \nabla v \, dV
	+\sum_{i=0}^{K-1} w_i^2 \mathcal{Q}_i(u) \mathcal{Q}_i(v)
	=
    \sum_{i=0}^{K-1} w_i^2 \mathcal{Q}_i(v) d_i
    , \quad \forall v \in U.
\end{equation}
\end{subequations}
where 
\begin{subequations}
\renewcommand{\theequation}{\theparentequation.\arabic{equation}}
	\begin{equation}
		U := \{v \in U_{ad} : \frac{\partial v}{\partial \nu} = 0 \text{ on } \Gamma_0\},
	\end{equation}
	\begin{equation}
		\Lambda = \{ \mu \in H^1(\Omega) : \frac{\partial \mu }{\partial \nu} = 0 \text{ on } \partial \Omega \}.
	\end{equation}
\end{subequations}

Let us rewrite \eqref{eq:variational_system1} as follows,
\begin{subequations}
\label{eq:variational_system2}
\renewcommand{\theequation}{\theparentequation.\arabic{equation}}
\begin{equation*}
	\text{\textit{Find triple }} 
	(f,\lambda,u)\in F \times \Lambda \times U
	\text{\textit{ which verifies }}
\end{equation*}
\begin{equation}
	a(f,\varphi) - b^t(\lambda, \varphi) = 0
	, \quad \forall \varphi \in F,
\end{equation}
\begin{equation}
	- b(f,\mu) + e^t(u,\mu) = 0
	, \quad \forall \mu \in \Lambda,
\end{equation}
\begin{equation}
	e(\lambda,v) + g(u,v) = R(v)
    , \quad \forall v \in U,
\end{equation}
\end{subequations}
where superscript $^t$ stands for transposition, bilinear forms $a(\cdot, \cdot)$, $b(\cdot, \cdot)$, 
$e(\cdot, \cdot)$, $g(\cdot, \cdot)$ and a linear form $R(\cdot)$
are defined as follows,
\begin{subequations}
\renewcommand{\theequation}{\theparentequation.\arabic{equation}}
\label{eq:bilinear_forms}
\begin{align}
	a(f, \varphi) = \varepsilon \int_{\Gamma_B} \nabla_B f \cdot \nabla_B \varphi \,dS , \\
	b(f,\mu) = \int_{\Gamma_B} f \mu \,dS , \\
	e(\lambda,v) = \int_{\Omega} \sigma \nabla \lambda \cdot \nabla v \, dV, \\
	g(u,v) = \sum_{i=0}^{K-1}w_i^2 \mathcal{Q}_i(u) \mathcal{Q}_i(v), \\
	R(v) = \sum_{i=0}^{K-1}w_i^2 \mathcal{Q}_i(v) d_i. 
\end{align} 
\end{subequations}
%

%%%%%%%%%%%%%%%%%%%%%%%%%%%%%%%%%%%%%%%%%%%%%%%%%%%%%%%
\section{Solvability}
\label{sec4}

In this section, we demonstrate that problem \eqref{eq:minimization} admits a unique solution.
In addition, since the formal Lagrangian method does constitute a rigorous proof, we prove that \eqref{eq:variational_system1} and \eqref{eq:minimization} are equivalent.

First, let us prove that problem \eqref{eq:minimization} has a unique optimal control $f$.
The general technique of proving the existence of the optimal control for 
a linear-quadratic problem is well established and based on the fact that $F$ 
is weak lower semi-continuous. 
Thus, we will check the required conditions.
Both $U_{ad}$ and $F$ are convex and closed.
We write problem \eqref{eq:minimization} as follows,
\begin{subequations}
\renewcommand{\theequation}{\theparentequation.\arabic{equation}}
\label{eq:linear-quadratic}
\begin{equation}
	\label{eq:linear-quadratic-1}
	\underset{f \in F, u \in U_{ad}}{\text{minimize}} \quad
	\Psi(f, u) :=
	\frac{1}{2} \norm{ \mathcal{Q} u - d}^2_2 + 
	\frac{\varepsilon}{2} \norm{ \nabla_B f}^2_{L^2(\Gamma_B)},
\end{equation}
\begin{equation}
	\label{eq:linear-quadratic-2}
	-\mydiv (\sigma \nabla u) = \xi, \text{ in } \Omega, 
\end{equation}
\begin{equation}
	\label{eq:linear-quadratic-3}
	 \sigma \frac{\partial u}{\partial \nu} = f, \text{ on } \Gamma_B,
\end{equation}
\begin{equation}
	\label{eq:linear-quadratic-4}
	 \frac{\partial u}{\partial \nu} = g_N, \text{ on } \Gamma_0,
\end{equation}
\end{subequations}
Here operator $\mathcal{Q}:\mathcal{C}(\Omega) \rightarrow \mathbb{R}^K$,
\begin{equation}
	\mathcal{Q} v := \{ \mathcal{Q}_0(v), .., \mathcal{Q}_{K-1}(v)\}.
\end{equation}
In what follows we will use its adjoint 
$\mathcal{Q}^*:\mathbb{R}^K \rightarrow \mathcal{C}(\Omega)^* $,  which is
\begin{equation}
	\mathcal{Q}^* z := \sum_{j=0}^{K-1} z_j \delta_j
\end{equation}
Functions $g_N$ and $\xi$ are given (precisely, $\xi=0$, $g_N=0$).
If $f$ is given as well, than the forward problem \eqref{eq:linear-quadratic-1}-\eqref{eq:linear-quadratic-4} is well-posed.
%Thus, we introduce operator $\mathcal{S}: H^1(\Gamma_B) \rightarrow H^1(\Omega)$ which maps a given $f$ to $u$.
Thus, we define the control-to-target operator, $\mathcal{G}:= f \mapsto d : H^1(\Gamma_B) \rightarrow \mathbb{R}^K$.
Operator $\mathcal{G}$ is linear.
Problem \eqref{eq:linear-quadratic} is equivalent to the reduced problem
\begin{equation}
	\label{eq:reduced-problem}
	\underset{f \in F}{\text{minimize}} \quad
	\Phi(f) :=
	\frac{1}{2} \norm{ \mathcal{G} f - d}^2_2 + 
	\frac{\varepsilon}{2} \norm{ \nabla_B f}^2_{L^2(\Gamma_B)}.
\end{equation}
The cost function $\Phi$ is convex.
Thus, we can formulate the following result.
\begin{proposition}
	For any $\varepsilon >0$ problem~\eqref{eq:linear-quadratic} admits the optimal control $\bar{f}$, 
	and optimal state $\bar{u}$, which are unique.
\end{proposition}
\begin{proof}
	\cite{Hinze2009}[Thm 1.43].
\end{proof}

Let us show that problems \eqref{eq:minimization} and \eqref{eq:variational_system1} are 
equivalent. 
%%%%%%%%%%%%%%%%%%%%%%%%%%%%%%%%%%%%%%%%%%%%%
\begin{proposition}
	Control $f$ and state  $u$ are the minimizers of \eqref{eq:minimization} if and only if there exists 
	$\lambda$ such that $(u,f,\lambda)$ solve variational system \eqref{eq:variational_system1}.
\end{proposition}
%%%%%%%%%%%%%%%%%%%%%%%%%%%%%%%%%%%%%%%%%%%%%
\begin{proof}

	We will assume $w_i^2=1$ (solely to render formulas more compact).
	
	%%%%%%%%%%%%%%%%%%%%
	$"\Rightarrow"$. 
	%%%%%%%%%%%%%%%%%%%%
	Let $(f,u,\lambda)$ be the stationary point of \eqref{eq:minimization}.

	Let $T: H^1(\Omega) \rightarrow L^2(\Gamma_B)$ be the trace operator,
	\begin{equation}
		Tu := \myat{u}_{\Gamma_B}, \quad v \in H^1(\Omega).
	\end{equation}
	Its adjoint $T^*:L^2(\Gamma_B) \rightarrow H^1(\Omega)^*$ is defined by relation 
	\begin{equation}
		\langle T^*v,u \rangle_{H^1(\Omega)^*,H^1(\Omega)} 
		= 
		\langle v,Tu \rangle_{L^2(\Gamma_B),L^2(\Gamma_B)}
		=		
		\int_{\Gamma_B} v \, Tu \, dS, \quad v \in L^2(\Gamma_B) , u \in H^1(\Omega).
	\end{equation}
	From \eqref{eq:variational_system1-2} we see that $u \in H^1(\Omega)$ is such that $\myat{\partial u / \partial \nu}_{\Gamma_0} = 0$ and
	\begin{equation}
		\int_{\Omega} \sigma \nabla u \cdot \nabla \mu \, dV  =
		\int_{\Gamma_B} f \, T \mu \, dS
		, \quad \forall \mu \in H^1(\Omega) ,
	\end{equation}
	The right-hand side is $\langle T \mu,f\rangle_{L^2(\Gamma_B),L^2(\Gamma_B)} = \langle \mu,T^* f\rangle_{H^1(\Omega),H^1(\Omega)^*}$.
	Thus, we can rewrite it as  
	\begin{equation}
		u = S T^* f,
	\end{equation}
	where $S:H^1(\Omega)^* \rightarrow H^1(\Omega)$.
	From \eqref{eq:variational_system1-3} it follows that 
	$\lambda \in H^1(\Omega)$, $\myat{\partial \lambda / \partial \nu}_{\partial \Omega} = 0$ and
	\begin{equation}
		\label{eq:proof-eq3}
		\int_{\Omega} \sigma \nabla \lambda \cdot \nabla v \, dV
		=
		\langle d - \mathcal{Q}(u) , \mathcal{Q}v \rangle_{\mathbb{R}^K,\mathbb{R}^K}.
	\end{equation}
	Since 
	$\langle d - \mathcal{Q}(u) , \mathcal{Q}v \rangle_{\mathbb{R}^K,\mathbb{R}^K}  = \langle Q^*(d - \mathcal{Q}(u)) , v \rangle_{H^1(\Omega)^*,H^1(\Omega)}$, we conclude that
	\begin{equation}
		\lambda = S^* \mathcal{Q}^* (Qu-d) = S^* \mathcal{Q}^* (QSTf-d),
	\end{equation}
	where
	$S^*: H^1(\Omega)^* \rightarrow H^1(\Omega)$.
	Let us take \eqref{eq:variational_system1-1},
	\begin{equation}	
		- \int_{\Gamma_B} T\lambda \, \varphi \, dS 
		+ \varepsilon \int_{\Gamma_B} \nabla_B f \cdot \nabla_B \varphi \, dS
		= 0
		, \quad \forall \varphi \in H^1(\Gamma_B).
	\end{equation}
	By substituting for $\lambda$ we get
	\begin{equation}	
		\label{eq:proof-eq4}	
		- \int_{\Gamma_B} T S^* \mathcal{Q}^* (Qu-d) \varphi \, dS 
		+ \varepsilon \int_{\Gamma_B} \nabla_B f \cdot \nabla_B \varphi \, dS
		= 0
		, \quad \forall \varphi \in H^1(\Gamma_B).
	\end{equation}
	Since $\nabla_f (Qu - d) = \nabla_f (QST^*f - d)=T S^* \mathcal{Q}^*$ we see that \eqref{eq:proof-eq4} means that the weak gradient of $\Phi$ with respect to $f$
	is zero. 
	Thus, $f$ minimizes $\Phi$, and $(f,u)$ minimizes \eqref{eq:minimization}.

	%%%%%%%%%%%%%%%%%%%%%%%%%%%%%%
	$"\Leftarrow"$.
	%%%%%%%%%%%%%%%%%%%%%%%%%%%%%%
	 
	Let $(u,f)$ be the minimizer of \eqref{eq:minimization}.
	It means that $u$ satisfies constraints \eqref{eq:minimization_2},\eqref{eq:minimization_2}, 
	e.i. $u \in H^1(\Omega)$, $\myat{\partial u / \partial \nu}_{\Gamma_0} = 0$ , and 
	\begin{equation}
		\label{eq:recovered1}
		\int_{\Omega} \nabla u \cdot \nabla v \, dV 
		= 
		\int_{\Gamma_B} f Tv \, dS, 
		\quad 
		\forall v \in H^1(\Omega), \myat{\partial v / \partial \nu}_{\Gamma_0} = 0.
	\end{equation}
	Thus we recovered equation \eqref{eq:variational_system1-2}.
	
	Let us introduce operators $S$ and $T$ with the same definitions that before.
	Equation \eqref{eq:recovered1} can be rewritten as 
	\begin{equation}
		u = S T^* f.
	\end{equation}
	Thus, 
	\eqref{eq:minimization} is equivalent to the reduced problem in the form
	\begin{equation}
		\label{eq:reduced-problem-2}
		\underset{f}{\text{minimize}} \quad
		\Phi(f) :=
		\frac{1}{2} \norm{ \mathcal{Q} S T^* f - d}^2_2 + 
		\frac{\varepsilon}{2} \norm{ \nabla_B f}^2_{L^2(\Gamma_B)},
	\end{equation}
	where $f \in F$.
	Function $f$ minimizes \eqref{eq:reduced-problem-2} and so it verifies
	the following equality,
	\begin{equation}
		\label{eq:12345}
		\int_{\Gamma_B} T S^* \mathcal{Q}^* (\mathcal{Q} S T^* f - d)  \varphi \, dS + 
		\varepsilon \int_{\Gamma_B} \nabla_B f \cdot \nabla_B \varphi \, dS = 0, 
		\quad 
		\forall \varphi \in F.
	\end{equation}
	We define 
	\begin{equation}
		\label{eq:lambda_def}
		\lambda := - S^* \mathcal{Q}^* (\mathcal{Q} u - d), 
	\end{equation}
	$\lambda \in H^1(\Omega)$, $\myat{\partial \lambda / \partial \nu}_{\partial \Omega} = 0$.
	Thus, \eqref{eq:12345} turns into
	\begin{equation}
		\label{eq:recovered2}
		-\int_{\Gamma_B} T \lambda  \varphi \, dS +
		\varepsilon \int_{\Gamma_B} \nabla_B f \cdot \nabla_B \varphi \, dS
		=0 , 
		\quad 
		\forall \varphi \in F,
	\end{equation}
	and we recovered equation \eqref{eq:variational_system1-1}.
	
	Finally, from \eqref{eq:lambda_def} of follows that $(S^*)^{-1} \lambda = -\mathcal{Q}^* (\mathcal{Q} u - d)$, 
	which is in the variational form reads
	\begin{equation}
		\label{eq:recovered3}
		\int_{\Omega} \nabla \lambda \cdot \nabla v \, dV
		= -\int_{\Omega} \mathcal{Q}^* (Qu-d) v \, dV,
		\quad 
		\forall v \in H^1(\Omega).
	\end{equation}
	Since $\int_{\Omega} \mathcal{Q}^* (Qu-d) v \, dV = \sum_{j=0}^{K-1} (Q_j v -d_j)\, Q_j v$, we recover 
	equation \eqref{eq:variational_system1-3}.

\end{proof}

%%%%%%%%%%%%%%%%%%%%%%%%%%%%%%%%%%%%%%%%%%%%%%%%%%%%%%%
\section{The finite-element discretization}
\label{sec5}

In this section, we formulate a discrete analog of \eqref{eq:variational_system2} using FEM.
Let us cover domain $\Omega$ with a set of tetrahedra, 
$\Omega_h = \bigcup_{i} \Omega_i$, where $h$ stands for the maximal diameter.
By $\mathcal{T}_h = \bigcup_i \mathcal{T}_i$ we will denote a set of faces of $\Omega_h$ on the cortical surface $\Gamma_B$.
We will assume that partitions $\Omega_h$ and $\mathcal{T}_h$ are regular, e.i. their elements 
will not degenerate as $h$ goes to zero.

By defining the finite-dimensional spaces $F_h \subset H^1(\Gamma_B)$, 
$U_h \subset H^1(\Omega)$,$\Lambda_h \subset H^1(\Omega)$,
we arrive to the following finite-dimensional problem,
\begin{subequations}
\label{eq:discrete_variational_system2}
\renewcommand{\theequation}{\theparentequation.\arabic{equation}}
\begin{equation*}
	\text{\textit{Find triple }} 
	(f_h,\lambda_h,u_h) \in F_h \times \Lambda_h \times U_h
	\text{\textit{ that verifies }}
\end{equation*}
\begin{equation}
	a_h(f_h,\varphi_h) - b_h^t(\lambda_h, \varphi_h) = 0
	, \quad \forall \varphi_h \in F_h,
\end{equation}
\begin{equation}
	- b_h(f_h,\mu_h) + e^t(u_h,\mu_h) = 0
	, \quad \forall \mu_h \in \Lambda_h,
\end{equation}
\begin{equation}
	e(\lambda_h,v_h) + g_h(u_h,v_h) = r_h(v_h)
    , \quad \forall v_h \in U_h,
\end{equation}
\end{subequations}

Specifically, in the numerical experiments that follow we define continuous piecewise-linear 
basis functions: $\{ \alpha_i\}_{i=0}^{N-1}$ on tetrahedral mesh $\Omega_h$ and $\{ \beta_i\}_{i=0}^{M-1}$ on triangular mesh $\mathcal{T}_h$.
That is, we set up the following function spaces,
\begin{subequations}
\renewcommand{\theequation}{\theparentequation.\arabic{equation}}
\label{eq:function_spaces}
	\begin{align}
		U_h = \Lambda_h := \mathrm{span}\{\alpha_0,..,\alpha_{N-1} \},   \\
		F_h := \mathrm{span}\{\beta_0,..,\beta_{M-1}\}.  
	\end{align}
\end{subequations}
The trial functions $u,\lambda \in H^1(\Omega)$, $f\in H^1(\Gamma_B)$ are approximated by 
functions $u_h \in U_h$, $\lambda_h \in \Lambda_h$, $f_h \in F_h$, respectively,
\begin{equation}
\label{eq:expansions}
		u_h = \sum_{i=0}^{N-1} u_i \alpha_i, \quad 
		\lambda_h = \sum_{i=0}^{N-1} \lambda_i \alpha_i,\quad 
		 f_h = \sum_{i=0}^{M-1} f_i \beta_i,\quad
\end{equation}
where $u_i,\lambda_i$, and $f_i$ are expansion coefficients (degrees of freedom, DOFs).
The test functions $v,\mu,\varphi$ are replaced by $v_h \in U_h$,$\mu_h \in \Lambda_h$ and $\varphi_h \in F_h$.
Thus, infinite-dimensional problem \eqref{eq:variational_system2} translates to the following 
algebraic problem,
\begin{equation}
\label{eq:fem_kkt}
	\underbrace{
		\left[
		\begin{array}{ccc}
			A  & -B^T & O  \\
			-B & O    &E^T \\
			O  & E    &G   \\
		\end{array}
		\right]
	}_{\mathcal{M}}	
	\underbrace{
		\left[
		\begin{array}{c}
			\tilde{f}  \\
			\tilde{\lambda}  \\ 
			\tilde{u}
		\end{array}
		\right]
	}_{\xi}		
		= 
	\underbrace{
		\left[
		\begin{array}{c}
			0  \\
			0  \\ 
			r
		\end{array}
		\right]
	}_{b}		
		,
\end{equation}
where 
\begin{subequations}
\label{eq:blocks_definition}
\renewcommand{\theequation}{\theparentequation.\arabic{equation}}
	\begin{align}
		A_{ij} = \varepsilon \int_{ \mathcal{T} } \nabla_B \beta_i \cdot \nabla_B \beta_j \, dS ,\\
		B_{ij} =             \int_{ \mathcal{T} } \beta_i \alpha_j \, dS ,\\
		E_{ij} =             \int_{ \Omega }  \sigma \alpha_i \alpha_j \, dV ,\\
		G = Q^T W^T W Q, \\
		r = Q^T W^T W d,
	\end{align}
\end{subequations}
where $Q$ is the matrix of observation operator $\mathcal{Q}_h:U_h \rightarrow \mathbb{R}^K$, $O$'s are zero blocks.
Vectors $\tilde{f}$, $\tilde{\lambda}$, and $\tilde{u}$ consist of DOFs $f_i,\lambda_i$, and $u_i$, respectively.

Blocks $A \in \mathbb{R}^{M\times M}$ and $E \in \mathbb{R}^{N\times N}$ are sparse positive-definite stiffness matrices. Matrix $B \in \mathbb{R}^{N\times M}$ is a sparse rectangular matrix of full column rank.
Matrix $G \in \mathbb{R}^{N \times N}$ is positive semidefinite. 
Its rank is equal to the number of electrodes minus 1 (we assume that there are no duplicate measurements).
In the presented form, $G$ is sparse. 
Moreover, if measuring electrodes are located at the mesh nodes, then $G$ will be diagonal.
However, one can replace $W^T W$ with a dense noise covariance matrix to account for correlated noise. 
In this case, block $G$ will be dense.
System matrix $\mathcal{M}$ is sparse symmetric indefinite.
Typically, $N >> M$, so the size of $\mathcal{M}$ is approximately $2N \times 2N$.

Note that a programming implementation of our method may be based on the standard procedures available in existing FEM libraries.
In particular, matrices $A$ and $E$ are nothing more than stiffness matrices; matrix $B$ is (a part of) a mass matrix.

%%%%%%%%%%%%%%%%%%%%%%%%%%%%%%%%%%%%%%%%%%%%%%%%%%%%%%%%%%%%
\section{FEM convergence}
\label{sec6}

In this section, we demonstrate that finite-dimensional solution $(u_h, f_h, \lambda_h)$
converges to the infinite-dimensional solution $(u, f, \lambda)$.
We introduce the linear operators $\mathbb{A}_h,\mathbb{B}_h, \mathbb{E}_h, \mathbb{G}_h$ associated 
with the bilinear forms $a_h(\cdot,\cdot),b_h(\cdot,\cdot),e_h(\cdot,\cdot),g_h(\cdot,\cdot)$, respectively.
By $\langle \cdot,\cdot \rangle$ we will denote the pairing between corresponding function spaces.

The following properties can be verified,
\begin{equation}
\label{eq:infsup}
	\beta \norm{\lambda_h} 
	\leq 
	\Norm{-\mathbb{B}_h^T\lambda_h}{\mathbb{E}_h\lambda_h}, 
	\quad \text{ (the inf-sup condition)}
\end{equation}
\begin{equation}
\label{eq:B_continuity}
	\Norm{-\mathbb{B}_h^T\lambda_h}{\mathbb{E}_h\lambda_h} \leq \beta_2 \norm{\lambda}, 
	\quad \text{ (continuity)}
\end{equation}
\begin{equation}
	\label{eq:ellipticity}
	\alpha_1 \Norm{f_h}{u_h} \leq \langle \mathbb{A}_h f_h, f_h \rangle + \langle \mathbb{G}_h u_h, u_h \rangle,  
	\quad \text{ (ellipticity) }
\end{equation}
\begin{equation}
	\label{eq:A_continuity}
	\Norm{\mathbb{A}_h f_h}{\mathbb{G}_h u_h} \leq \alpha_2 \Norm{f_h}{u_h}.  \quad \text{ (continuity) }
\end{equation}

The following result bounds the norm of approximation error $(u-u_h, f-f_h, \lambda - \lambda_h)$ 
by the norm of the interpolation error.
Here we use the method of proving the estimates developed in \cite{Brezzi1974}.

\begin{proposition}

Let $(f,u,\lambda) \in F \times U \times \Lambda$ be a solution of \eqref{eq:variational_system2}, 
and $(f_h,u_h,\lambda) \in F_h \times U_h \times \Lambda_h$ is a solution to \eqref{eq:discrete_variational_system2}.
Then the misfit between the two solutions satisfies to the following estimates,
	\begin{subequations}
	\renewcommand{\theequation}{\theparentequation.\arabic{equation}}
		\label{eq:basic_estimate}
		\begin{equation}
			\label{eq:basic_estimate.1}
			\begin{split}
				\Norm{f-f_h}{u-u_h}
				\leq 
				\left(  1+\frac{\alpha_2}{\alpha_1}  \right)
				\left(  1+\frac{\beta_2}{\beta_1}  \right)
				\inf_{\varphi_h,v_h \in F_h \times U_h}
				\Norm{f-\varphi_h}{u-v_h}.
			\end{split}
		\end{equation}
		\begin{equation}
			\label{eq:basic_estimate.2}
			\begin{split}
				\norm{\lambda - \lambda_h} 
				\leq
				\left(  1+\frac{\beta_2}{\beta_1} \right)
				\inf_{\mu_h \in \Lambda_h}
				\norm{\lambda - \mu_h}
				\\				
				+ 
				\frac{\alpha_2}{\beta} 
				\left(  1+\frac{\alpha_2}{\alpha_1}  \right) 
				\left(  1+\frac{\beta_2}{\beta_1}  \right) 
				\inf_{\varphi_h,v_h \in F_h \times U_h}
				\Norm{f-\varphi_h}{u-v_h}			
			\end{split}
		\end{equation}
	\end{subequations}
\end{proposition}

\begin{proof}

The error $(f-f_h,u-u_h,\lambda-\lambda_h)$ satisfies the following set of equations,
\begin{subequations}
\renewcommand{\theequation}{\theparentequation.\arabic{equation}}
\label{eq:2x2_err}
	\begin{equation}
		\label{eq:2x2_err.1}
		a_h(f-f_h, \varphi_h) + g_h(u-u_h,v_h)
		-b_h^t(\lambda-\lambda_h,\varphi_h) + e_h(\lambda-\lambda_h,v_h)
		= 0 , \quad \forall \varphi_h, v_h \in F_h \times U_h,
	\end{equation}
	\begin{equation}	
		\label{eq:2x2_err.2}
		-b_h(f-f_h,\mu_h) + e_h^t(u-u_h,\mu_h) = 0, \quad \forall \mu_h \in \Lambda_h.
	\end{equation}
\end{subequations}

Let us define the space of all functions $(\varphi,v) \in F_h \times U_h$ verifying equation \eqref{eq:2x2_err.2},
	\begin{equation}
		K 
		:= \{
			\varphi \times v \in F_h \times U_h :
		-b_h (f-\varphi, \mu_h ) + e_h^t (u-v, \mu_h )
		= 0, \quad \forall \mu_h \in \Lambda_h
		\}
		,
	\end{equation}
%.
Similarly, we define a space of all functions $\mu \in \Lambda_h$ satisfying the transposed equation \eqref{eq:2x2_err.2},
	\begin{equation}
		K^*
		:= \{
			\mu  \in  \Lambda_h:
		-b_h^t(\lambda-\mu,\varphi_h ) +e_h(\lambda-\mu, v_h)
		=0,	\quad \forall \varphi_h, v_h \in F_h \times U_h
		\}
		.
	\end{equation}
Now we choose arbitrary $(f_*,v_*) \in K$ and $\lambda_* \in K^*$.
Equations \eqref{eq:2x2_err} turn into
\begin{subequations}
\renewcommand{\theequation}{\theparentequation.\arabic{equation}}
\label{eq:2x2_diff}
	\begin{equation}
		\label{eq:2x2_diff.1}
		\begin{split}
			a_h (f_h-f_*, \varphi_h ) + g_h (u_h-u_*,v_h)
			-b_h^t (\lambda_h-\lambda_*,\varphi_h ) + e_h (\lambda_h-\lambda_*, v_h)
			= 
			\\
			a_h (f-f_*, \varphi_h ) + g_h (u-u_*,v_h)
			\quad \forall \varphi_h, v_h \in F_h \times U_h,
		\end{split}
	\end{equation}
	\begin{equation}	
		\label{eq:2x2_diff.2}
		-b_h (f_h-f_*, \mu_h ) + e_h^t (u_h-u_*, \mu_h )
		= 0, \quad \forall \mu_h \in \Lambda_h.
	\end{equation}
\end{subequations}
From the whole set of admissible $(\varphi_h,v_h, \mu_h)$ we select the very special one,
$\varphi_h = f_h-\varphi_*$, $v_h = u_h-v_*$, $\mu_h = \lambda_h-\lambda_*$. 
With this choice, it follows from \eqref{eq:2x2_diff},
	\begin{equation}
		\begin{split}		
			a_h(f_h-\varphi_*,f_h-\varphi_*) + g_h(u_h-v_*,u_h-v_*)
			=
			a_h(f-\varphi_*,f_h-\varphi_*) + g_h(u-v_*,u_h-v_*).
		\end{split}
	\end{equation}
Thus, 
	\begin{equation}
		\begin{split}		
			\norm{f_h-\varphi_*}^2 + \norm{u-v_*}^2
			\leq 
			\frac{1}{\alpha_1}
			\Big( 
				a_h(f_h-\varphi_*,f_h-\varphi_*) + g_h(u_h-v_*,u_h-v_*)
			\Big)
			= \\
			\frac{1}{\alpha_1}
			\Big( 
				a_h(f_h-\varphi_*,f-\varphi_*) + g_h(u_h-v_*,u-v_*)
			\Big)			
			\leq  \\
			\frac{\alpha_2}{\alpha_1} \Norm{f-\varphi_*}{u-v_*} \Norm{f_h-\varphi_*}{u_h-v_*},
		\end{split}
	\end{equation}
and 
	\begin{equation}
			\Norm{f_h-\varphi_*}{u_h-v_*}
			\leq 
			\frac{1}{\alpha_1}
			\frac{\alpha_2}{\alpha_1} \Norm{f-\varphi_*}{u-v_*}.
	\end{equation}
	From the triangle inequality it follows 
	\begin{equation}
			\Norm{f-f_h}{u-u_h}
			\leq 
			\left(  1+\frac{\alpha_2}{\alpha_1}  \right)
			\Norm{f-\varphi_*}{u-v_*}.
	\end{equation}

From \eqref{eq:2x2_diff.1} it follows that 
	\begin{equation}
		-b_h^t (\lambda_h-\lambda_*,\varphi_h ) + e_h (\lambda_h-\lambda_*, v_h) = \\
		a_h (f-f_h, \varphi_h ) + g_h (u-u_h,v_h).
	\end{equation}
Thus, 
	\begin{equation}
		\begin{split}		
			\norm{\lambda_h - \lambda_*} \leq \\ 
			\frac{1}{\beta_1} \Norm{-\mathbb{B}_h(\lambda_h-\lambda_*)}{\mathbb{E}_h(\lambda_h-\lambda_*)}
			\leq \\
			\frac{1}{\beta_1} \Norm{-\mathbb{A}_h(f_h-f_*)}{\mathbb{G}_h(u_h-u_*)}
			\leq \\
			\alpha_2 \Norm{f-f_h}{u-u_h}
		\end{split}
	\end{equation}
and so	
	\begin{equation}
		\begin{split}		
			\norm{\lambda - \lambda_h} \leq
			\norm{\lambda - \lambda_*} + \norm{\lambda_h - \lambda_*}
			\leq \\
			\norm{\lambda - \lambda_*} + 
			\frac{\alpha_2}{\beta_1} \left(  1+\frac{\alpha_2}{\alpha_1}  \right) \Norm{f-f_*}{u-u_*}			
		\end{split}
	\end{equation}

Since $(f_*,u_*,\mu_*)$ is arbitrary, we have the following estimates,
	\begin{subequations}
	\renewcommand{\theequation}{\theparentequation.\arabic{equation}}
		\label{eq:basic_estimate_K}
		\begin{equation}
			\label{eq:basic_estimate_K.1}
				\Norm{f-f_h}{u-u_h}
				\leq 
				\left(  1+\frac{\alpha_2}{\alpha_1}  \right)
				\inf_{\varphi_h \times v_h \in K}
				\Norm{f-\varphi_h}{u-v_h}.
		\end{equation}
		\begin{equation}
			\label{eq:basic_estimate_K.2}
			\begin{split}
				\norm{\lambda - \lambda_h} 
				\leq
				\inf_{\mu_h \in K^*}
				\norm{\lambda - \mu_h} + 
				\frac{\alpha_2}{\beta_1} \left(  1+\frac{\alpha_2}{\alpha_1}  \right) 
				\inf_{\varphi_h \times v_h \in K}
				\Norm{f-\varphi_h}{u-v_h}			
			\end{split}
		\end{equation}
	\end{subequations}
Now, proceeding as in \cite{Braess}[Remark 4.10] we can bound the infimum over $K$ by the infimum over $\mathbb{G}_h \times U_h$,
	\begin{equation}
		\label{eq:basic_estimate_K.1}
		\inf_{\varphi_h, v_h \in K}
			\Norm{f-f_h}{u-u_h}
			\leq 
			\left(  1+\frac{\beta_2}{\beta_1}  \right)
			\inf_{\varphi_h \times v_h \in F_h \times U_h}
			\Norm{f-\varphi_h}{u-v_h}.
	\end{equation}
Similarly, 
	\begin{equation}
		\label{eq:basic_estimate_K.1}
		\inf_{\lambda_h \in K^*}
			\norm{\lambda-\lambda_h}
			\leq 
			\left(  1+\frac{\beta_2}{\beta_1}  \right)
			\inf_{\lambda_h \in \Lambda_h}
			\norm{\lambda-\lambda_h}.
	\end{equation}
Thus, the estimates \eqref{eq:basic_estimate.1} and \eqref{eq:basic_estimate.2} follow.

\end{proof}

This result shows that the norm of approximation error is bounded by the norm of the interpolation error,
\begin{equation}
	\label{eq:final_estimate}
	\begin{split}
		\NormTri{f-f_h}{u-u_h}{\lambda-\lambda_h}
		\leq 
		C
		\inf_
		{
			\substack{\varphi_h \in F_h \\ v_h \in U_h \\ \lambda_h \in \Lambda_h}
		}
		\NormTri{f-\varphi_h}{u-v_h}{\lambda-\mu_h}.
	\end{split}
\end{equation}
We use linear node-based elements on tetrahedral grids for all of $f,u,\lambda$.
We apply the classical results on finite-element approximation in Sobolev spaces \cite{Ciarlet}
to \eqref{eq:final_estimate} to get the following estimate.

\begin{proposition}
Let $(f,u,\lambda) \in F \times U \times \Lambda$ be a solution of \eqref{eq:variational_system2}, 
and $(f_h,u_h,\lambda) \in F_h \times U_h \times \Lambda_h$ is a solution to \eqref{eq:discrete_variational_system2}.
Then the error satisfies
\begin{equation}
	\begin{split}
		\NormTri{f-f_h}{u-u_h}{\lambda-\lambda_h} 
		\leq \\
		C h
		\left(
		|f|^2_{H^2(\Gamma_B)}
		+
		|u|^2_{H^2(\Omega)}
		+
		|\lambda|^2_{H^2(\Omega)}
		\right)^{1/2},
	\end{split}
\end{equation}
where $|\cdot|$ stands for a corresponding seminorm, $C$ is a constant.

\end{proposition}

%

%%%%%%%%%%%%%%%%%%%%%%%%%%%%%%%%%%%%%%%%%%%%%%%%%%%%%%%%%%%
\section{Numerical experiments}
\label{sec7}

In this section, we demonstrate that the solution of variational system \eqref{eq:variational_system2} 
provide a viable estimate of cortical activity.
Our programming implementation is based on C++ finite-element library MFEM \citep{mfem}.
The matrix factorization was performed by a sparse direct solver UMFPACK \citep{UMFPACK}.

In the first series of calculations, we reconstructed the electric potential in a spherical shell with unit conductivity.
The inner and outer radii were $0.7$ and $1.0$, respectively.
The shell was split into $86K$ tetrahedra.
The potential inside the shell was represented as a finite series of spherical harmonics.
The coefficients of the series were selected so that $u$ satisfies Laplace's equation in the shell, condition $\frac{\partial u}{\partial \nu}=0$ on the outer sphere, and $u=\xi$ on the inner sphere, were $\xi$ is some predefined distribution of potential.
See \citep{Malovichko2020} for more details.
Potential on the inner shell formed a cross of unit amplitude (Figure~\ref{fig:spherical_shell_true}c).
Potential on the outer shell is depicted in Figure~\ref{fig:spherical_shell_true}a.
Synthetic EEG data were sampled at 198 electrodes evenly scattered on the hemisphere $y>0$ (Figure~\ref{fig:spherical_shell_true}b).

In the first numerical experiment, we perform source reconstruction for several values of $\varepsilon$ from $10^{-12}$ to $10^{-7}$.
The data were noise-free.
As expected, the smaller $\varepsilon$, the rougher cortical activity is reconstructed, see Figure~\ref{fig:spherical_shell}.
We also observe that as $\varepsilon$ is getting smaller, the changes in the estimated cortical activity are less visible.
In particular, the reconstructed values for $\varepsilon=10^{-11}$ 
and $\varepsilon=10^{-12}$ are very close to each other.
The norm of the data residual (the difference between the input and calculated data) is plotted in Figure~\ref{fig:rmse}.
\begin{figure}[!htb]
\centering
	\begin{tabular}{ccc}
		\includegraphics[width=2.in]{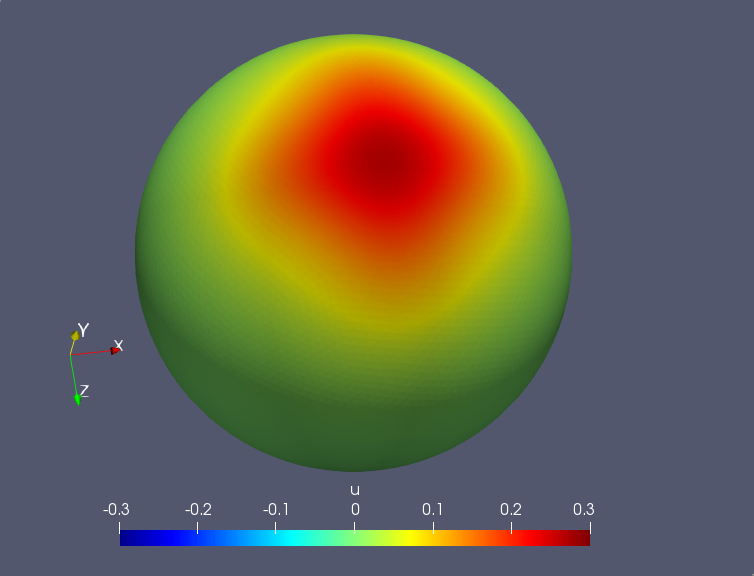} &
		\includegraphics[width=2.in]{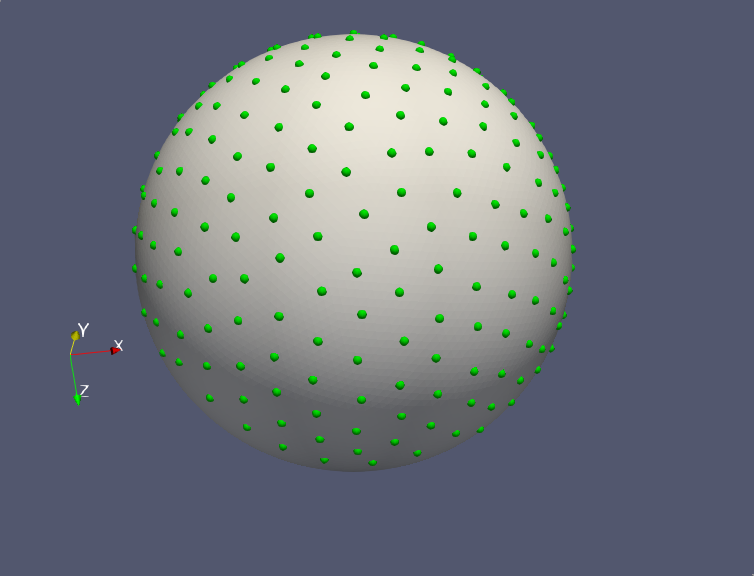} &
		\includegraphics[width=2.in]{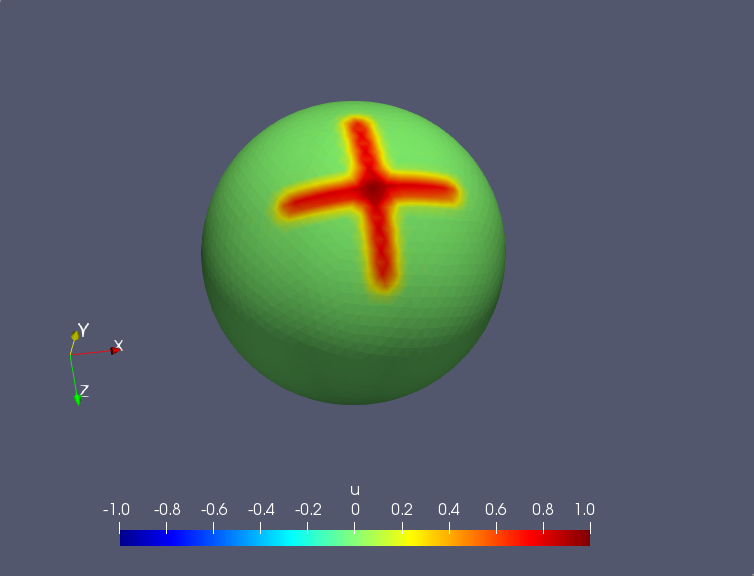} \\
		(a) & 
		(b) & 
		(c) \\
	\end{tabular}
	\caption
	{
		Synthetic EEG data in a spherical shell.
		(a) Electric potential on the outer sphere (radius 1.0).
		(b) 198 electrodes distributed on the outer half sphere $y>0$.
		(c) Electric potential on the inner sphere (radius 0.7).
	}
\label{fig:spherical_shell_true}
\end{figure}
\begin{figure}[!htb]
\centering
	\begin{tabular}{ccc}
		\includegraphics[width=1.5in]{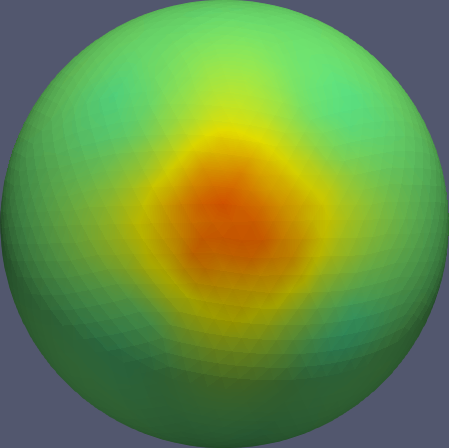} &
		\includegraphics[width=1.5in]{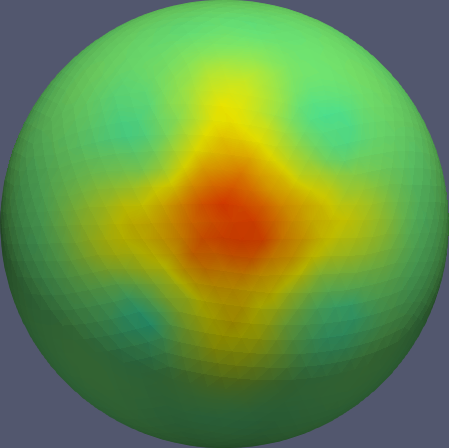} &
		\includegraphics[width=1.5in]{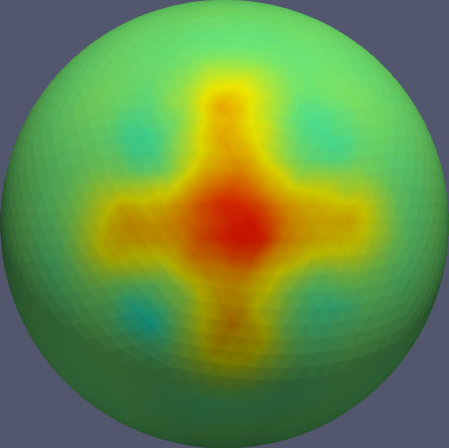} \\
		(a) $\varepsilon=10^{-7}$ & 
		(b) $\varepsilon=10^{-8}$ & 
		(c) $\varepsilon=10^{-9}$ \\
		\includegraphics[width=1.5in]{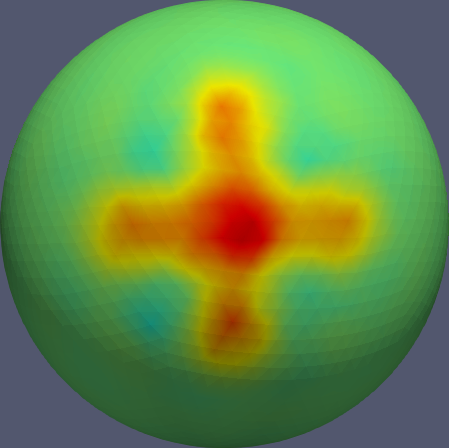} &
		\includegraphics[width=1.5in]{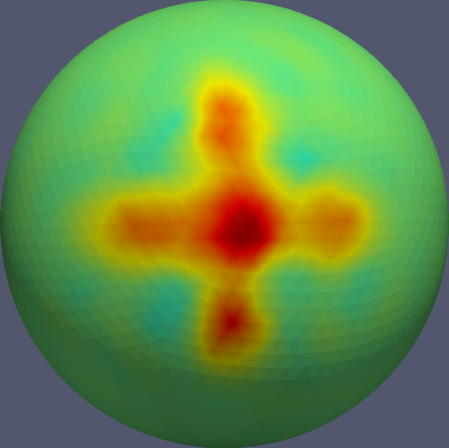} &
		\includegraphics[width=1.5in]{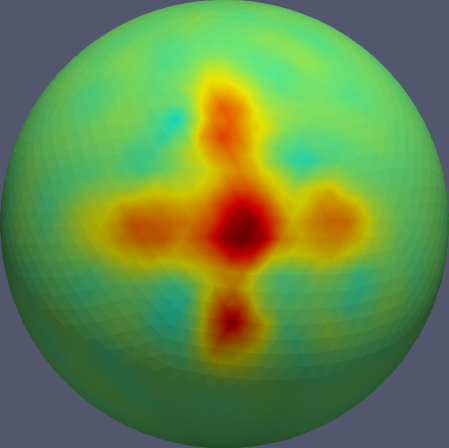} \\
		(d) $\varepsilon=10^{-10}$ & 
		(e) $\varepsilon=10^{-11}$ & 
		(f) $\varepsilon=10^{-12}$ \\
	\end{tabular}
	\caption
	{
		Electric potential reconstructed on the inner boundary of the spherical shell for different values of regularization parameter $\varepsilon$.
		The color scale is the same as in Figure~\ref{fig:spherical_shell_true}c.
		The layout consisted of 198 electrodes evenly distributed on the outer half-sphere centered at the cross.
		The input data were noise-free.
	}
\label{fig:spherical_shell}
\end{figure}
\begin{figure}[!htb]
\centering
		\includegraphics[width=4.in]{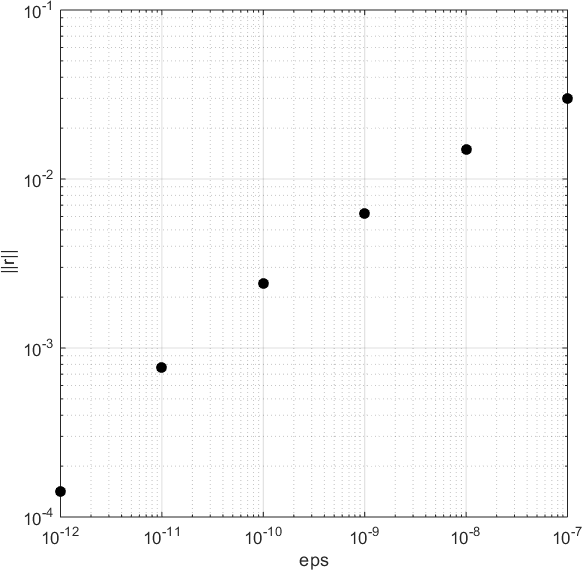}
	\caption
	{
		The norm of the data residual $||r||$ as a function of $\varepsilon$.
	}
\label{fig:rmse}
\end{figure}

\newpage

In the next numerical experiment, we added noise to the synthetic EEG data.
At each electrode position, the noise was Gaussian with the zero mean and the standard deviation equal to $1\%$ of the amplitude of the true EEG signal at that point.
Resulting reconstructions are presented in Figure~\ref{fig:spherical_shell_noise0.01}.
The reconstructed distributions remained very similar to those without noise, although the images appeared distorted due to the impact of the noise.
To evaluate the data fit, we computed the root-mean-square error (RMSE) as
\begin{equation}
	e = \sqrt{    \frac{1}{K}\sum_{i=0}^{K-1} \frac{(d_i-\hat{u}_i)^2}{s^2_i}  }.
\end{equation}
Here $s_i$ are the estimated values of the standard deviation of the noise at the $i$-th electrode, 
$\hat{u}_i$ are the calculated values of electric potential at the end of reconstruction. 
Value $e=1$ means that the data fit is close to optimal: on average, it corresponds to the noise level.
The plot of $e$ of $\varepsilon$ for the $1\%$ noise is given in Figure~\ref{fig:rmse_noise0.01}.
In this sense, the optimal value of regularization is $\varepsilon=10^{-9}$.
Below that level, the data are over-fitted, and so the image of cortical activity on the cortex gets progressively distorted.

\begin{figure}[!htb]
\centering
	\begin{tabular}{ccc}
		\includegraphics[width=1.5in]{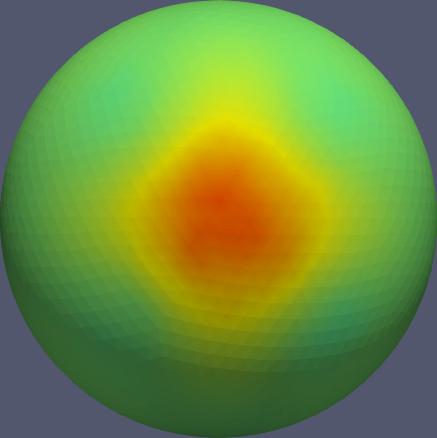} &
		\includegraphics[width=1.5in]{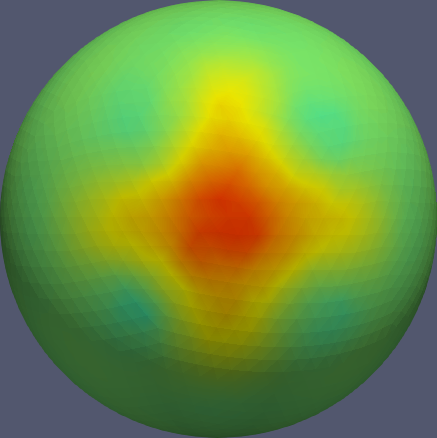} &
		\includegraphics[width=1.5in]{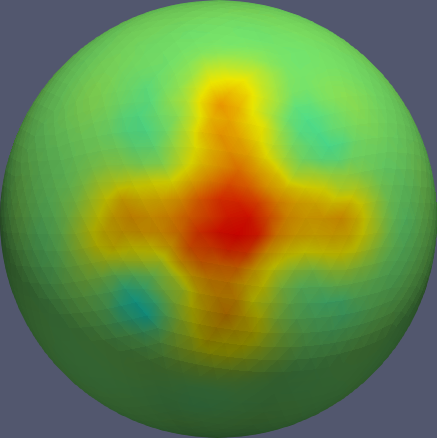} \\
		(a) $\varepsilon=10^{-7}$ & 
		(b) $\varepsilon=10^{-8}$ & 
		(c) $\varepsilon=10^{-9}$ \\
		\includegraphics[width=1.5in]{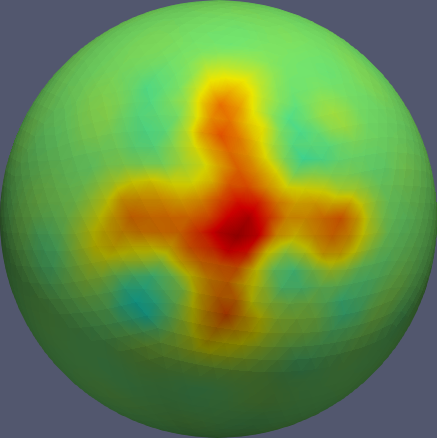} &
		\includegraphics[width=1.5in]{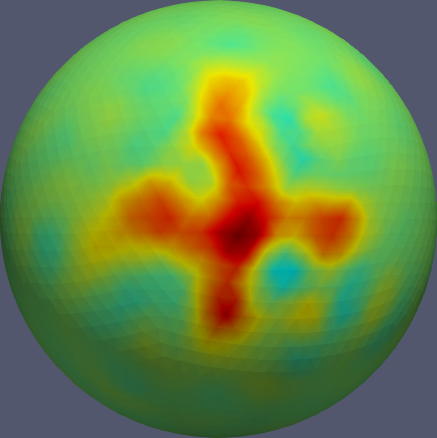} &
		\includegraphics[width=1.5in]{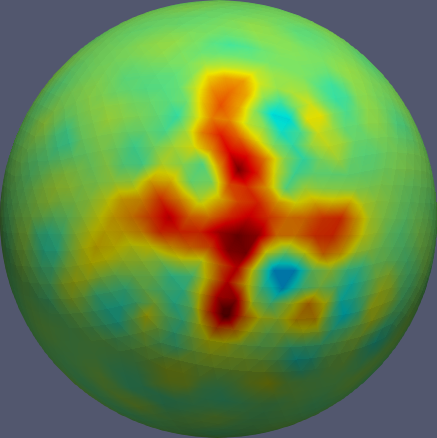} \\
		(d) $\varepsilon=10^{-10}$ & 
		(e) $\varepsilon=10^{-11}$ & 
		(f) $\varepsilon=10^{-12}$ \\
	\end{tabular}
	\caption
	{
		Electric potential reconstructed on the inner boundary of the spherical shell from the data contaminated with $1\%$ noise.
		The color scale is the same as in Figure~\ref{fig:spherical_shell_true}c.
	}
\label{fig:spherical_shell_noise0.01}
\end{figure}

\begin{figure}[!htb]
\centering
		\includegraphics[width=4.in]{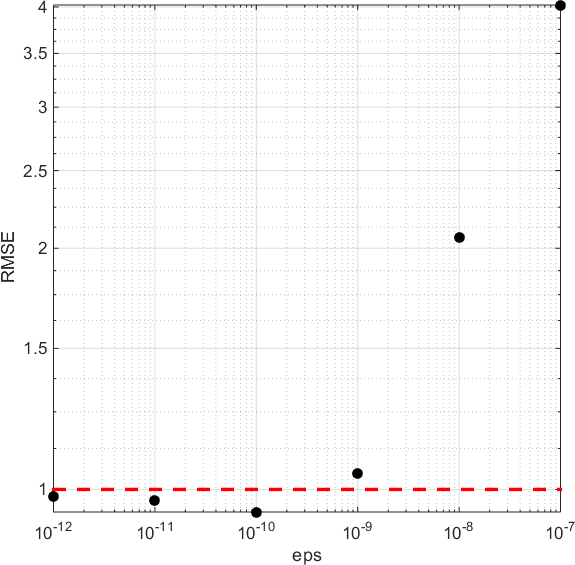}
	\caption
	{
		The root-mean-square error $e$ of data contaminated with the $1\%$ noise as a function of $\varepsilon$.
		The dashed red line marks the value $e$=1.
	}
\label{fig:rmse_noise0.01}
\end{figure}

\newpage

We repeated the same calculation with the $5\%$ noise. 
The reconstructed distribution of $u$ is presented in Figure~\ref{fig:spherical_shell_noise0.05}.
The misfit is plotted in Figure~\ref{fig:rmse_noise0.05}.
We observe that a higher regularization is needed for the $5\%$ noise compared to the $1\%$ one.
The spatial resolution of the reconstructed distributions as lower, as expected.
\begin{figure}[!htb]
\centering
	\begin{tabular}{ccc}
		\includegraphics[width=1.5in]{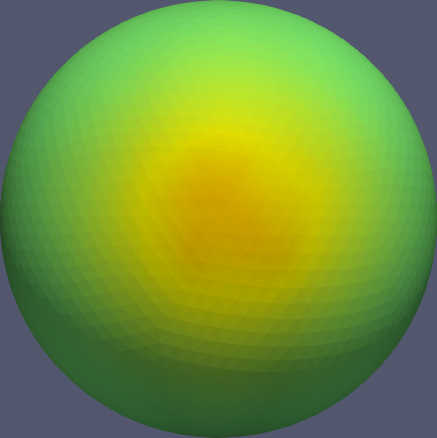} &
		\includegraphics[width=1.5in]{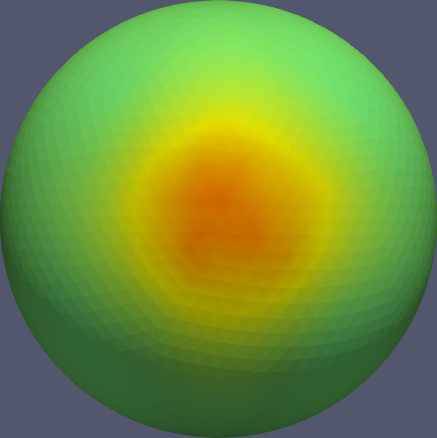} &
		\includegraphics[width=1.5in]{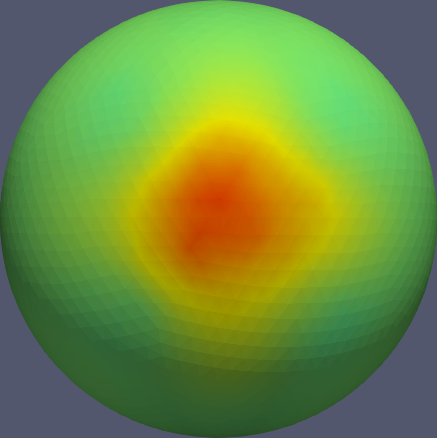} \\
		(a) $\varepsilon=10^{-5}$ & 
		(b) $\varepsilon=10^{-6}$ & 
		(c) $\varepsilon=10^{-7}$ \\
		\includegraphics[width=1.5in]{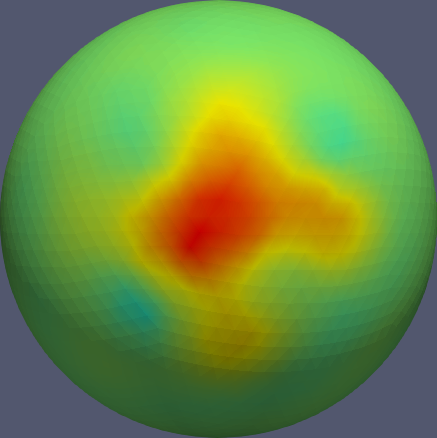} &
		\includegraphics[width=1.5in]{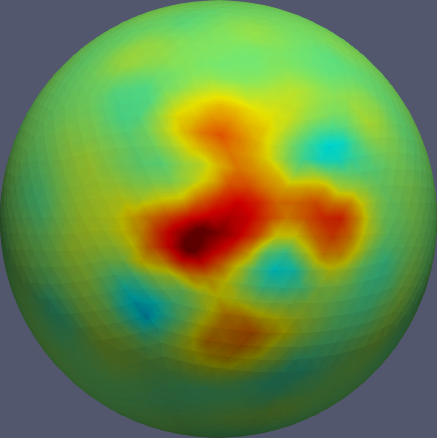} &
		\includegraphics[width=1.5in]{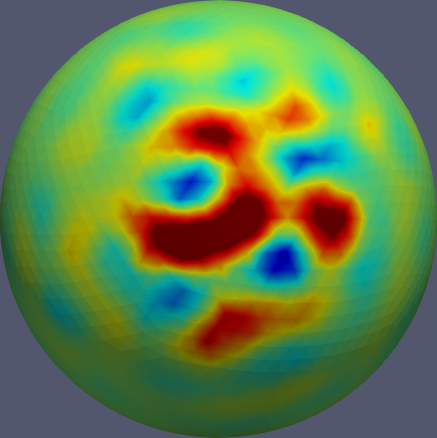} \\
		(d) $\varepsilon=10^{-8}$ & 
		(e) $\varepsilon=10^{-9}$ & 
		(f) $\varepsilon=10^{-10}$ \\
	\end{tabular}
	\caption
	{
		Electric potential reconstructed on the inner boundary of the spherical shell from the data contaminated with $5\%$ noise.
		The color scale is the same as in Figure~\ref{fig:spherical_shell_true}c.
	}
\label{fig:spherical_shell_noise0.05}
\end{figure}
\begin{figure}[!htb]
\centering
		\includegraphics[width=4.in]{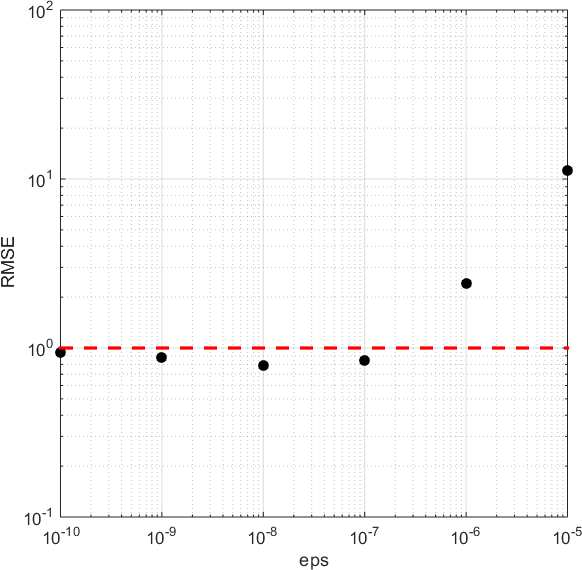}
	\caption
	{
		The root-mean-square error $e$ of data contaminated with the $5\%$ noise as a function of $\varepsilon$.
		The dashed red line marks the value $e$=1.
	}
\label{fig:rmse_noise0.05}
\end{figure}

\newpage

As proof of the concept, we applied our algorithm to a realistic head model.
We employed the mesh described in \citep{Malovichko2020}.
The model consisted of four compartments to which the following conductivities were assigned: 
skin 0.33 S/m, skull 0.011 S/m, CSF 1.0 S/m, and brain 0.33 S/m.
The mesh contained 587,882 tetrahedra.
We set up two active patches in the visual cortex containing 146 unit dipoles.
The dipoles were located in the barycenters of the tetrahedra. 
Orientations of the dipoles were normal to the cortex.
The forward problem was solved by FEM using the standard $H^1$-conforming finite elements of the first order.
Synthetic EEG data were obtained by sampling the potential at 175 electrodes extracted from 
256-electrode Geodesic Sensor Net montage by Philips.
We adjusted electrode locations to make them coincide with some grid vertices.
The mesh, brain activations, simulated potential, and electrode locations are depicted in Figure~\ref{fig:experiment1}.

The mesh used for reconstruction was created from the same surfaces, but the brain was replaced with a void.
We changed meshing parameters to make the mesh differ from the original one, thus avoiding the so-called inverse crime.
The final grid consisted of 658,513 tetrahedra.
Reconstruction was performed using the two values of $\varepsilon$ : $10^{-3}$ and $10^{-4}$.
Resulted distributions of potential on the cortex are given in Figure~\ref{fig:result1}.

The reconstructed cortical activity is consistent with the location of the actual sources.  
Predictably, the reconstruction for bigger $\varepsilon$ is more smeared.

\begin{figure*}
  \centering
  \begin{tabular}[b]{lll}
    \includegraphics[height=0.3\textwidth]{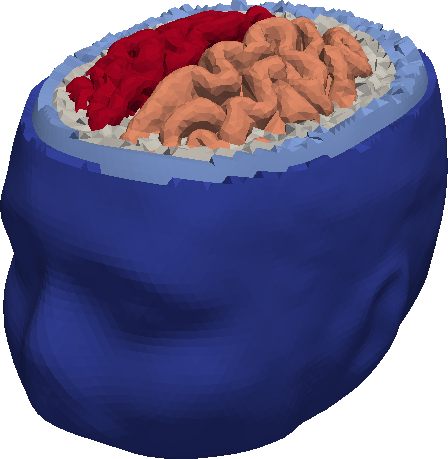} & &
    \includegraphics[height=0.25\textwidth]{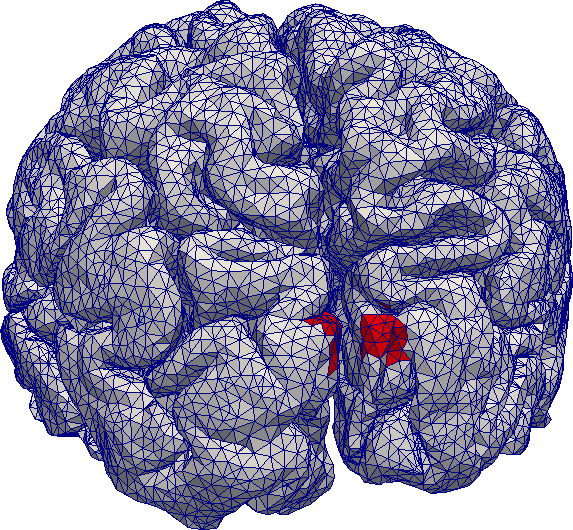} 
	\\
    \multicolumn{1}{c}{(a)}	& &	\multicolumn{1}{c}{(b)}
	\\
    \includegraphics[height=0.35\textwidth]{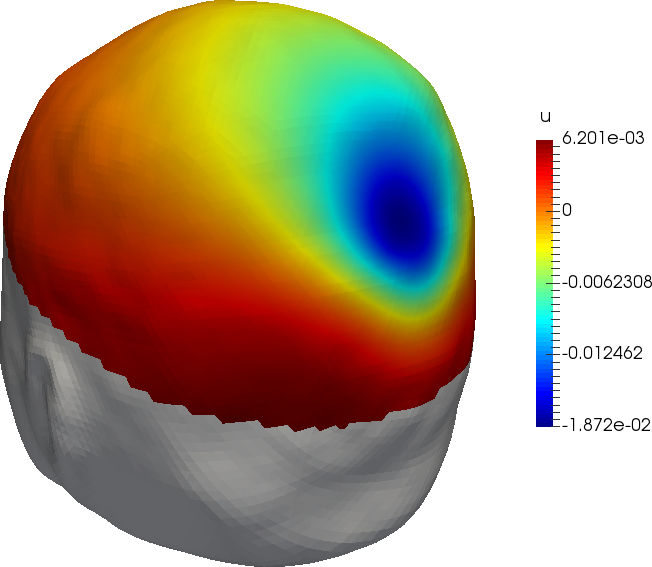} & &
    \includegraphics[height=0.35\textwidth]{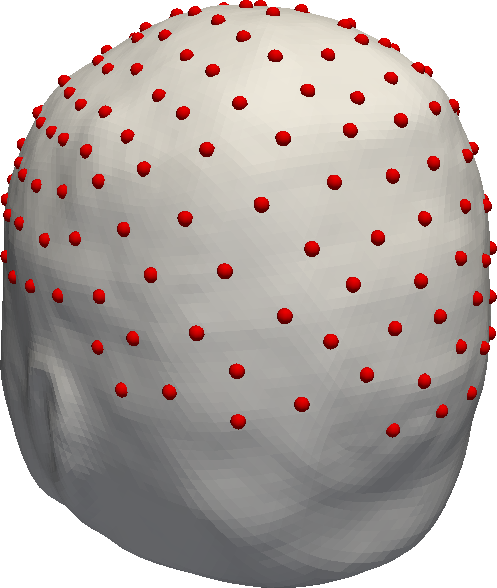} 
	\\
    \multicolumn{1}{c}{(c)}& & \multicolumn{1}{c}{(d)}
  \end{tabular}
    \caption
	{
		Simulated EEG data used in the numerical experiment.
		Panel (a). Tetrahedral  mesh used for simulating EEG data. 
		Panel (b). Prescribed activations in the visual cortex.
		Panel (c). Simulated potential on the scalp.
		Panel (d). 175 electrodes extracted from a 256-electrode montage.
    }
    \label{fig:experiment1}
\end{figure*}

\begin{figure*}
  \centering
  \begin{tabular}[b]{c}
    \includegraphics[width=0.45\textwidth]{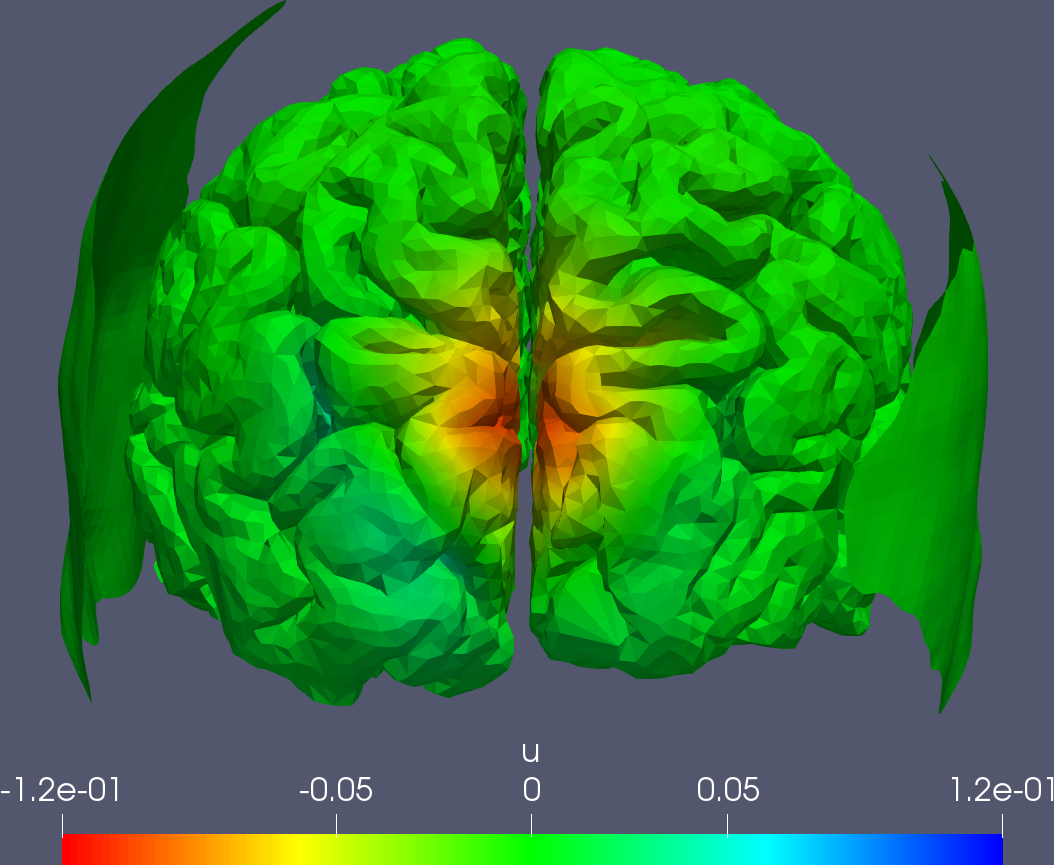} \\
    \small (a)
  \end{tabular}
  \begin{tabular}[b]{c}
    \includegraphics[width=0.45\textwidth]{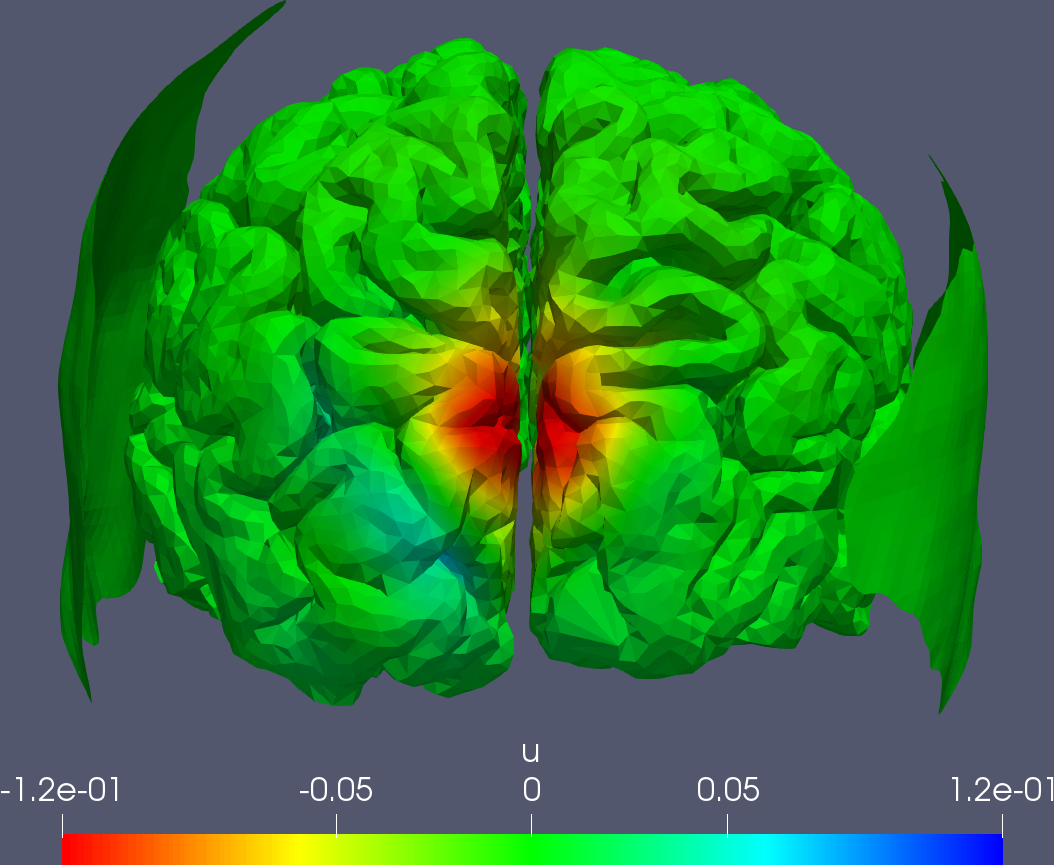}\\
    \small (b)
  \end{tabular}
    \caption{
    Potential $u$ recovered at the cortex. Regularization parameter was set to $\varepsilon=10^{-3}$ and $\varepsilon=10^{-4}$ on panels (a) and (b), respectively.
    }
    \label{fig:result1}
\end{figure*}

%%%%%%%%%%%%%%%%%%%%%%%%%%%%%%%%%%

\newpage

\section{Conclusion}
\label{sec8}

In this article, we applied OC formalism to the inverse EEG problem. 
The inverse problem is reduced to a system of variational equations giving rise to a system of linear equations. 
This system is a $3 \times 3$ block matrix. 
It is large symmetric indefinite, and sparse.
Thus, estimation of the EEG sources is performed by a single linear solve. 
We presented a numerical test that demonstrates that this approach provides a consistent estimation of cortical activity.
We demonstrated that our algorithm is stable in the presence of noise in the input data.

Our approach fully addresses the critical shortcomings of the mixed QRM: discrete data and noise characteristics are naturally taken into account when solving the inverse problem.
In contrast to the standard methods of the MNE/LORETA type, which estimate the power of individual dipoles, our approach reconstructs the density of the dipole layer located on the cortex. 
On the one hand, this makes our approach more specialized. 
On the other, it reduces the mesh size and eliminates the need to distinguish between gray and white matter in the model because
the area occupied by the brain is excluded from consideration. 
Also, our approach avoids the problem of simulation of the dipolar right-hand side - a persisting issue with all standard FEM-based methods of MNE/LORETA type.
In other respects, our method offers the same flexibility as the FEM-based estimators of MNE/LORETA type, such as the ability to work with anatomical head models of almost arbitrary complexity, non-constant conductivity within head compartments, holes in the interfaces. 
Finally, we note that the software implementation of this approach does not require special efforts since it uses programming blocks that are available in standard FEM packages.

\section*{Acknowledgments}

This research was partially supported by the Russian Science Foundation, project no. 21-11-00139.
The authors acknowledge computational resources granted by 
Complex for Simulation and Data Processing for Mega-Science Facilities at NRC “Kurchatov Institute”, http://ckp.nrcki.ru/.

%%%%%%%%%%%%%%%%%%%%%%%%
\clearpage

%\section*{References}

\newcommand{\newblock}{} %required for natbib iopart compatibility
\bibliographystyle{dcu}
\bibliography{srcinv-eeg-JIIP_v2}

\end{document}